\documentclass[12pt]{amsart}
\usepackage{amsmath,amstext,amsbsy,amssymb,amsfonts,amscd,calc}
\usepackage[colorlinks=true]{hyperref}
\usepackage[utf8]{inputenc}
\usepackage{multirow}
\usepackage{fullpage}
\usepackage{tikz}
\usetikzlibrary{arrows}

\newtheorem{conjecture}{Conjecture}[section]
\newtheorem{theorem}[conjecture]{Theorem}
\newtheorem{lemma}[conjecture]{Lemma}
\newtheorem{corollary}[conjecture]{Corollary}
\newtheorem{proposition}[conjecture]{Proposition}
\theoremstyle{definition}
\newtheorem{definition}[conjecture]{Definition}
\newtheorem{example}[conjecture]{Example}
\newtheorem{remark}[conjecture]{Remark}
\newtheorem{algorithm}[conjecture]{Algorithm}

\newcommand{\Z}{{\mathbb{Z}}}

\newcommand{\Q}{{\mathbb{Q}}}

\newcommand{\Qp}{{\Q_p}}
\newcommand{\calO}{{\mathcal{O}}}
\newcommand{\KK}{{K}}

\newcommand{\clK}{{\overline{K}}} % algebraic closure
\newcommand{\OK}{{\calO_\KK}} % valuation ring
\newcommand{\RK}{{\underline{\KK}}} % residual class field
\newcommand{\repsK}{R_\RK}  % fixed set of representatives of the elements of $\RK$ in $\OK$
\newcommand{\reps}[1]{R_{#1}}  % fixed set of representatives
\newcommand{\repss}[1]{\reps{#1}^\times}  % fixed set of representative
\newcommand{\RA}{\underline{A}} % residual (or associated) polynomial
\newcommand{\KL}{{L}}
\newcommand{\OL}{{\calO_\KL}}

\newcommand{\KN}{{N}}                 % the field N
\newcommand{\KM}{{M}}                 % the field M
\newcommand{\slopeden}{e} %  numerator of lambda
\newcommand{\slopenum}{h} %  denominator of lambda
\newcommand{\invA}{\mathcal{A}} % unnamed invariant A
\newcommand{\npol}{\mathcal{N}} % newton polygon
\newcommand{\rpol}{\mathcal{R}} % ramification polygon
\newcommand{\spol}{\mathcal{S}} % a segment
\newcommand{\ppol}{\mathcal{P}} % set of points of a polygon, segment, component

\newcommand{\rpolypoints}{\{(1,J_0),(p^{s_1},J_1),\ldots,(p^{s_{u-1}},J_{u-1}),(p^{s_u},0),\ldots,(n,0)\}}
\usepackage{enumitem}
\setlist[itemize,1]{label=$\bullet$}
\setlist[itemize,2]{label=$\circ$}
\setlist[itemize,3]{label=$\circ$}
\setlist[itemize,4]{label=$\bullet$}
\setlist[itemize,5]{label=$\bullet$}
\setlist[itemize,6]{label=$\bullet$}
\setlist[itemize,7]{label=$\bullet$}
\setlist[itemize,8]{label=$\bullet$}
\setlist[itemize,9]{label=$\bullet$}

\renewlist{itemize}{itemize}{9}

\setlist[enumerate,1]{label={\rm(\alph*)}}
\setlist[enumerate,2]{label={\rm(\roman*)}}

\newcommand{\invset}{X} % set of invariants
\newcommand{\PP}{\mathfrak{p}}        % Prime ideal of p-adic field

\begin{document}

\title{Counting Extensions of $\mathfrak{p}$-adic Fields with Given Invariants}
\author{Brian Sinclair}
\begin{abstract}
We give two specializations of Krasner's mass formula.
The first formula yields the number of extensions of
a $\mathfrak{p}$-adic field with given, inertia degree, ramification index, discriminant, and
ramification polygon. We then refine this formula further to the case where an additional invariant related to the
residual polynomials of the segments of the ramification polygon is given.
\end{abstract}

\maketitle

\section{Introduction}

Krasner \cite{krasner} gave a formula for the number of totally ramified extensions, using his famous lemma as a main tool.  In addition to the choice of degree, his formula depends on the choice of discriminant.
This choice allows the construction of a finite set of Eisenstein polynomials which
generate all totally ramified extensions of given discriminant.

We specialize these methods to compute the number of totally ramified
extensions with the additional choice of ramification polygon and residual polynomials 
of segments, using the results on these invariants from \cite{pauli-sinclair}.
The descriptions of these invariants allow us to compute all of their possible values,
as is done in \cite{sinclair-phd} and partition the extensions of particular degree
and discriminant, the results of which we see in our examples in Section \ref{sec counting ex}.

\subsection{Notation}

We denote by $\Qp$ the field of $p$-adic numbers and by $v_p$ the (exponential) valuation
with $v_p(p)=1$.
Let $\KK$ be a finite extension of $\Qp$, let $\OK$ be the valuation ring of 
$\KK$, and denote by $\pi$ a uniformizer of $\OK$.
Let $\mathfrak{p} = (\pi)$ be the ideal generated by $\pi$.

For $\gamma\in\OK$ we denote by $\underline\gamma$ the class $\gamma+\mathfrak{p}$ in $\RK=\OK/\mathfrak{p}$.
We denote by $\repsK$ a fixed set of representatives of the elements of $\RK$ in $\OK$ and
if $\underline\delta\in\RK$ then $\widehat\delta\in\repsK$ is the element such
$\underline{\widehat\delta}=\underline\delta$.

We write $v_\PP$ for the valuation of $\KK$ that is normalized such that $v_\PP(\pi)=1$.
The unique extension of $v_\PP$ to an algebraic closure $\clK$ of $\KK$
(or to any intermediate field) is also denoted by $v_\PP$.
For $\alpha\in\KK$, we write $|\alpha|$ to denote its absolute value.

For a polynomial $f\in\OK[x]$ of degree $n$ we denote its coefficients by $f_i$ ($0\le i\le n$)
such that $f(x)=x^n+f_{n-1}x^{n-1}+\dots+f_0$ and write 
$ f_i = \sum_{j=0}^\infty f_{i,j} \pi^j $ where $f_{i,j}\in\repsK$.
If $f$ is Eisenstein then $f_n=1$, $f_{0,1}\ne 0$ and $f_{i,0}=0$ for $0\leq i \leq n-1$.

We will use the following notation to refer to elements in $\KK$
with finitely many non-zero $\pi$-adic coefficients.
Let $l,m$ be two integers with $1 \leq l \leq m$,
let $\reps{l,m}$ be a fixed set of representatives of the quotient $\PP^l / \PP^m$,
and let $\repss{l,m}$ be the subset of $\reps{l,m}$ whose elements have $\pi$-adic valuation of exactly $l$.

%%%%%%%%%%%%%%%%%%%%%%%%%%%%%%%%%%%%%%%%%%%%%%%%%%%%%%%%%%%%%%%%%%%%%%%%%%%%%%

\section{Effects of Extension Invariants on Generating Polynomials}

\subsection{Discriminant}

We recall some of the results Krasner used to obtain his formula for the
number of extensions of a $p$-adic field \cite{krasner},
which may also be found in \cite{pauli-roblot}.
The possible discriminants of finite extensions are given by Ore's conditions \cite{ore-bemerkungen}.

\begin{proposition}[Ore's conditions]\label{prop.ore}
Let $\KK$ be a finite extension of $\Qp$ with maximal ideal
$\PP$. Given $J_0 \in \Z$ let $a,b\in\Z$ be such that $J_0 = a_0n + b_0$ and
$0 \leq b_0 < n$. Then there exist totally ramified extensions
$\KL/\KK$ of degree $n$ and discriminant $\PP^{n+J_0-1}$ if and only if
\[ \min\{v_\PP(b_0) n, v_\PP(n) n\} \leq J_0 \leq v_\PP(n) n. \]
\end{proposition}

The proof of Ore's conditions yields a certain form for the generating polynomials 
of extensions with given discriminant.

\begin{lemma}\label{lem ell}
Let $J_0 = a_0n+b_0$ satisfy Ore's conditions.
For $1\leq i \leq n-1$, let
\[
l(i)=\left\{\begin{array}{ll}
\max\{ 2 + a_0 - v_\PP(i), 1\} & \mbox{ if } i < b_0, \\
\max\{ 1 + a_0 - v_\PP(i), 1\} & \mbox{ if } i \geq b_0. \\
\end{array}\right. 
\]
An Eisenstein polynomial $f\in\OK[x]$ 
has discriminant $\PP^{n+J_0-1}$ % where $J_0=a_0n+b_0$ with $0\le b_0<n$ fulfills Ore's conditions
if and only if $v_\PP(f_i) \geq l(i)$ and, if $b_0\neq 0$, $v_\PP(f_{b_0}) = l(b_0)$.
\end{lemma}

Krasner's Lemma yields a bound over which the $\PP$-adic coefficients of the
coefficients of a generating polynomial can be chosen to be $0$, while still 
generating the same extensions.
We state the following lemma here, but its proof will be shown as a
consequence of Theorem \ref{thm discs-D}.

\begin{lemma}[Krasner]\label{lem krasner bound}
Each totally ramified extension of degree $n$ with discriminant $\PP^{n+J_0-1}$
can be generated by a polynomial with coefficients in $\reps{0,c}$, where $c > 1 + (2J_0)/n$.
\end{lemma}

By the previous two lemmas, every extension of a given discriminant can be generated
by a polynomial of a particular form with coefficients with only finitely many
non-zero $\PP$-adic coefficients.

\begin{proposition}\label{prop psi-disc}
Let $J_0 = a_0n+b_0$ satisfy Ore's conditions, $c > 1 + 2a_0 + \frac{2b_0}{n}$, $l(i)$ as in Lemma \ref{lem ell},
and let $\Psi_{n,J_0}(c)$ be the set of all polynomials $\psi(x)=x^n + \sum \psi_i x^i \in \OK[x]$ with
\[
\psi_i \in \left\lbrace \begin{array}{ll}
\repss{1,c}    & \mbox{ if } i = 0\\
\repss{l(i),c} & \mbox{ if } i = b_0 \neq 0 \\
\reps{l(i),c}  & \mbox{ if } 1\leq i\leq n-1 \mbox{ and } i \neq b_0. \\
\end{array}\right.
\]
The polynomials in $\Psi_{n,J_0}(c)$ are Eisenstein polynomials of discriminant $\PP^{n+J_0-1}$,
and each totally ramified extension of degree $n$ with discriminant $\PP^{n+J_0-1}$
can be generated by a polynomial in $\Psi_{n,J_0}(c)$.
\end{proposition}

%%%%%%%%%%%%%%%%%%%%%%%%%%%%%%%%%%%%%%%%%%%%%%%%%%%%%%%%%%%%%%%%%%%%%

\subsection{Ramification Polygons}

To distinguish totally ramified extensions further we use an additional invariant, namely the {ramification polygon}.  

\begin{definition}\label{def ram pol}
Assume that the Eisenstein polynomial $f$ defines $\KL/\KK$.  The \emph{ramification polygon} $\rpol_f$ of $f$ is the Newton polygon $\mathcal{N}$ of the \emph{ramification polynomial} $\rho(x)=f(\alpha x + \alpha)/(\alpha^n)\in  K(\alpha)[x]$ of $f$, where $\alpha$ is a root of $f$.
\end{definition}
The ramification polygon $\rpol_f$ of $f$ is an invariant of $\KL/\KK$ (see \cite[Proposition 4.4]{greve-pauli} for example) called the ramification polygon of $\KL/\KK$ which we denote $\rpol_{\KL/\KK}$.
Ramification polygons have been used to study ramification groups and reciprocity \cite{scherk},
compute splitting fields and Galois groups \cite{greve-pauli},
describe maximal abelian extensions \cite{lubin},
and answer questions of commutativity in $p$-adic dynamical systems \cite{li}.

Let $f(x)=\sum_{i=0}^{n}f_ix^i\in\KK[x]$ be an Eisenstein polynomial, denote by $\alpha$ a root of $f$, and set
$\KL=\KK(\alpha)$. Let $\rho(x)=\sum_{i=0}^{n}\rho_i x^i\in\KL[x]$ be the ramification polynomial of $f$.
Then the coefficients of $\rho$ are
\[
\rho_i=\sum_{k=i}^n\binom{k}{i}\; f_k\; \alpha^{k-n} 
\]
As $v_\alpha(\alpha)=1$ and $v_\alpha(f_i)\in n\Z$ we obtain
\begin{equation}\label{eq val coeff rho}
v_\alpha(\rho_i) = \min_{i\leq k \leq n} \left\lbrace v_\alpha \left( \binom{k}{i}\; f_k\; \alpha^k \right) - n\right\rbrace
                 = \min_{i\leq k\leq n}\left\lbrace n \left[ v_\PP \left( \binom{k}{i}\; f_k \right) - 1\right] + k \right\rbrace.
\end{equation}

\begin{remark}
Throughout this paper we describe ramification polygons by the set of points 
$\ppol=\rpolypoints$ where not all points in $\ppol$ have to be vertices 
of the polygon $\rpol$.
We write $\rpol=\ppol$.
This gives a finer distinction between fields by their ramification polygons and also 
allows for an easier description of the invariant based on the residual 
polynomials of the segments of the ramification polygon, see Section \ref{sec res seg}.
\end{remark}

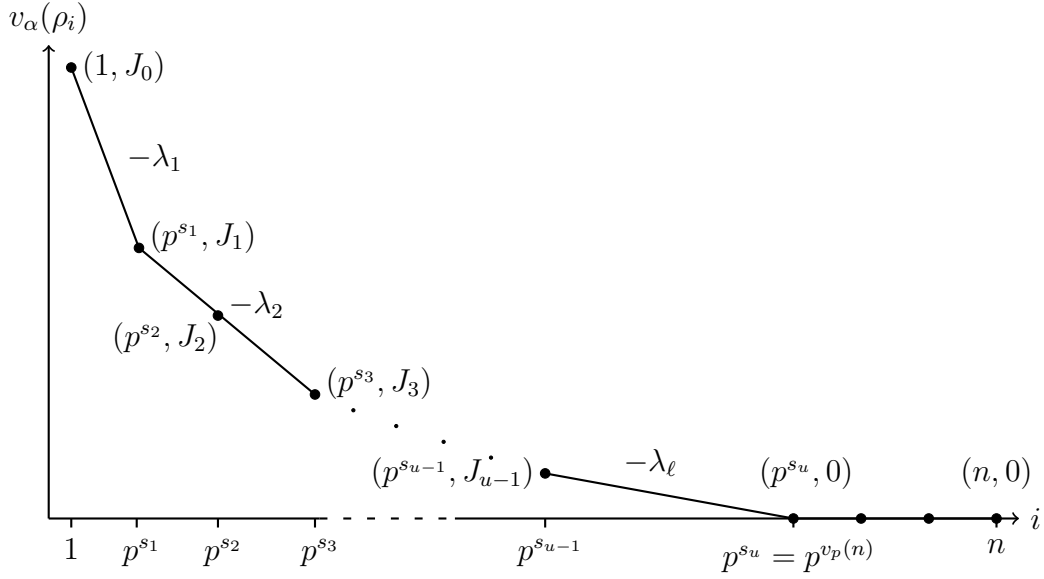
\begin{figure}
\begin{center}
\begin{tikzpicture}[scale=0.6]
\draw[thick] (-1,0) -- (5,0);
\draw[thick, loosely dashed] (5,0) -- (8,0);
\draw[thick,->] (8,0) -- (20.5,0) node[anchor=west] {$i$};
\draw[thick,->] (-1,0) -- (-1,10.5) node[anchor=south] {$v_\alpha(\rho_i)$};

\draw[thick] (-0.5,10) -- (1,6) -- (4.9,2.75) ;
\draw[thick] (10,1) -- (15.5,0) -- (20,0);
\filldraw[black] (-0.5,10) circle (3pt)
                 (1,6) circle (3pt)
                 (2.75,4.5) circle (3pt)
                 (4.9,2.75) circle (3pt)
                 (10,1) circle (3pt)
                 (15.5,0) circle (3pt)
                 (20,0) circle (3pt);
\filldraw[black] (5.75,2.4) circle (1pt)
                 (6.7,2.05) circle (1pt)
                 (7.75,1.7) circle (1pt)
                 (8.8,1.35) circle (1pt);
\filldraw[black] (17,0) circle (3pt)
                 (18.5,0) circle (3pt);

\foreach \x/\xtext in {-0.5/1, 0.95/p^{s_1}\!\!, 2.75/p^{s_2}\!\!, 4.9/p^{s_3}\!\!, 10/p^{s_{u-1}}\!\!, 15.5/p^{s_u}=p^{v_p(n)}\!\!, 20/n}
  \draw[thick] (\x,0) -- (\x,-0.2) node[anchor=north] {$\xtext$};

\draw (0.5,8)     node[anchor=west] {$-\lambda_1$}
      (2.75,4.75) node[anchor=west] {$-\lambda_2$}
      (11.5,1.25) node[anchor=west] {$-\lambda_\ell$}
      (-0.5,10)   node[anchor=west] {$(1,J_0)$}
      (1,6.25)    node[anchor=west] {$(p^{s_1},J_1)$}
      (3.0,4)     node[anchor=east] {$(p^{s_2},J_2)$}
      (4.9,3)     node[anchor=west] {$(p^{s_3},J_3)$}
      (10,1)      node[anchor=east] {$(p^{s_{u-1}},J_{u-1})$}
      (14.5,1)    node[anchor=west] {$(p^{s_u},0)$}
      (20,1)      node[anchor=center] {$(n,0)$} ;
\end{tikzpicture}
\end{center}

\caption{Ramification polygon of an Eisenstein polynomial $f$ of degree $n$ and discriminant $\PP^{n+J_0-1}$
with $\ell+1$ segments and $u-1$ points on the polygon with ordinate above $0$.
}\label{fig ram pol}
\end{figure}

We recall a necessary result concerning the points of ramification polygons
and their effects on generating polynomials from \cite[Section 3]{pauli-sinclair}.

\begin{lemma}\label{lem vf}
Let $f\in\OK[x]$ be an Eisenstein polynomial with ramification polygon
\[\rpol_f = \rpolypoints\]
where $J_i = a_in+b_i$ with $0\leq b_i\leq n-1$.
Then for $0\leq t\leq u$ we have
\[v_\PP(f_i) \geq \left\{\begin{array}{ll}
  2+a_t-v_\PP \binom{i}{p^{s_t}} \mbox{ for } p^{s_t} \leq i < b_t\\
  1+a_t-v_\PP \binom{i}{p^{s_t}} \mbox{ for } b_t \leq i \leq n-1\\
  \end{array}\right. ,\]
and $v_\PP(f_{b_t})=a_t+1-v_\PP \binom{b_t}{p^{s_t}}$ if $b_t\neq 0$.
Further, if there is no point with abscissa $p^i$, where $s_t < i < s_{t+1}$ for some $1\le t\le u$, then for $k$ such that $p^i\leq k \le n-1$,
\[
v_\PP(f_k) > \frac{1}{n}\left[ \frac{J_{t+1}-J_t}{p^{s_{t+1}}-p^{s_t}} (p^i-p^{s_t})+J_t-k \right]+1-v_\PP\binom{k}{p^i}
\]
\end{lemma}

A single point on the ramification polygon can give a lower bound for the valuation of multiple coefficients,
so we define functions that give us the minimum valuation of a coefficient based on a point (or lack thereof) on the ramification polygon with abscissa $p^s$ and minimum valuations based on the entire polygon.

In the following definition, due to Lemma \ref{lem vf},
$l_{\rpol_f}(i,s)$ for $1\leq s\leq s_u$ and $p^s\leq i\leq n$ gives the minimum valuation of $f_i$
due to a point (or lack thereof) with abscissa $p^s$ on the ramification polygon $\rpol_f$ of $f$.
By taking the maximum of these over all $s$, we define $L_{\rpol_f}(i)$ so that
$v_\PP(f_i)\ge L_{\rpol_f}(i)$ for $1\le i \le n-1$.

\begin{definition}\label{def ells}
Let $\rpol_f$ be the ramification polygon of $f$ with points
\[ \rpol_f = \rpolypoints, \]
and where $J_i = a_in+b_i$ with $0\leq b_i\leq n-1$. For $0\leq t\leq u$, let
\[
l_{\rpol_f}(i,s_t)=\left\{\begin{array}{ll}
\max\{ 2 + a_t - v_\PP \binom{i}{p^{s_t}}, 1\} & \mbox{ if } p^{s_t} \leq i < b_t, \\
\max\{ 1 + a_t - v_\PP \binom{i}{p^{s_t}}, 1\} & \mbox{ if } i \geq b_t. \\
\end{array}\right. 
\]
If there is no point above $p^u$ with $s_t<u<s_{t+1}$, then for $p^u\leq i\leq n-1$, let
\[
l_{\rpol_f}(i,u) = \max\left\lbrace
\left\lceil \frac{1}{n}\left[ \frac{J_{t+1}-J_t}{p^{s_{t+1}}-p^{s_t}} (p^u-p^{s_t})+J_t-k \right]+1-v_\PP\binom{k}{p^u} \right\rceil
%\left\lceil\frac{1}{n}\left[ \frac{J_{t+1}-J_t}{p^{s_{t+1}-s_t}-1} (p^{u-s_t}-1)+J_t-i \right]+1-v_\PP\binom{i}{p^u} \right\rceil
,1\right\rbrace
\]
Finally, set
\[
L_{\rpol_f}(i) = 
\left\{\begin{array}{ll}
1 &\mbox{if } i=0\\
\max \{l_{\rpol_f}(i,t) : p^t \leq i\} &\mbox{if }1\leq i\leq n-1\\
0 &\mbox{if } i=n
\end{array}\right..
\]
\end{definition}

Definition \ref{def ells} yields a simple way to describe a set of conditions for an Eisenstein polynomial to generate an extension with given ramification polygon $\rpol$.

\begin{proposition}\label{prop R iff}
An Eisenstein polynomial $f$ has ramification polygon
\[\rpol = \rpolypoints,\]
where $J_i = a_in+b_i$ with $0\leq b_i\leq n-1$,
if and only if
\begin{enumerate}
\item $v_\PP(f_i) \geq L_\rpol(i)$
\item For $0\leq t\leq u$, $v_\PP(f_{b_t})=L_\rpol(b_t)$ if $b_t\neq 0$.
\end{enumerate}
where $L_\rpol$ is as defined in Definition \ref{def ells}.
\end{proposition}

Similar to the case of the discriminant, we have an function for the lower
bound of the valuation of each coefficient, and so we can construct a finite
set of Eisenstein polynomials with given ramification polygon as a consequence
of Proposition \ref{prop R iff} and Lemma \ref{lem krasner bound}.

\begin{proposition}\label{prop psi-rpol}
For a ramification polygon $\rpol$ with points $(p^0,J_0),(p^{s_1},J_1),\dots,\,(p^{s_\ell},J_\ell),$
where $J_i = a_in+b_i$, let $B_\rpol$ be the set of non-zero $b_i$.
Let $c > 1 + 2a_0 + \frac{2b_0}{n}$,
and let $\Psi_{n,J_0,\rpol}(c)$ be the set of all polynomials $\psi(x)=x^n + \sum \psi_i x^i \in \OK[x]$ with
\[
\psi_i \in \left\lbrace \begin{array}{ll}
\repss{1,c}          & \mbox{ if } i = 0\\
\repss{L_\rpol(i),c} & \mbox{ if } i \in B_\rpol \\
\reps{L_\rpol(i),c}  & \mbox{ if } 1\leq i\leq n-1 \mbox{ and } i \notin B_\rpol \\
\end{array}\right.
\]
The polynomials in $\Psi_{n,J_0,\rpol}(c)$ generate all totally ramified extensions of $\KK$ of degree $n$, discriminant $\PP^{n+J_0-1}$, and ramification polygon $\rpol$.
\end{proposition}

%%%%%%%%%%%%%%%%%%%%%%%%%%%%%%%%%%%%%%%%%%%%%%%%%%%%%%%%%%%%%%%%%%%%%

\subsection{Residual Polynomials of Segments}\label{sec res seg}

Residual (or associated) polynomials were introduced by Ore \cite{ore-newton}. 
They yield information about the unramified part of the extension generated by the factors of a polynomial. This makes them a useful tool in the computation of ideal decompositions and integral bases \cite{guardia-montes-nart,montes,montes-nart} the closely related problem of polynomial factorization over local fields \cite{guardia-nart-pauli, pauli-pf2},
and the efficient enumeration of totally ramified local field extensions \cite{pauli-sinclair}.

\begin{definition}[Residual polynomial]\label{def res poly}
Let $\KL$ be a finite extension of $\Qp$ with uniformizer $\alpha$.  
Let  $\rho(x)=\sum_i \rho_i x^i\in\OL[x]$.
Let $\spol$ be a segment of the Newton polygon of 
$\rho$  of length $l$ with
endpoints $(k,v_\alpha(\rho))$ and $(k+l,v_\alpha(\rho_{k+l}))$, and slope $-\slopenum/\slopeden=\left(v_\alpha(\rho_{k+l})-v_\alpha(\rho_k)\right)/l$
then
\[
\RA(x)=\sum_{j=0}^{l/e}\underline{\rho_{je+k}\alpha^{jh-v_\alpha(\rho_k)}}x^{j}\in\RK[x]
\]
is called the \emph{residual polynomial} of $\spol$.
\end{definition}

\begin{remark}
The ramification polygon of a polynomial $f$ and the residual polynomials of its segments yield a subfield $\KM$ of the splitting field $\KN$ of 
$f$, such  that $\KN/\KM$ is a $p$-extension \cite[Theorem 9.1]{greve-pauli}. 
\end{remark}

From the definition we obtain some of the properties of residual polynomials.

\begin{lemma}\label{lem res pol}
Let $\KL$ be a finite extension of $\Qp$ with uniformizer $\alpha$.  Let $\rho\in\OL[x]$.
Let $\npol$ be the Newton polygon of $\rho$  with segments $\spol_1,\dots,\spol_\ell$ and let
$\RA_1,\dots,\RA_\ell$ be the corresponding residual polynomials.
\begin{enumerate}
\item \label{lem res pol int}
If $\spol_i$ 
has integral slope $-\slopenum\in\Z$
with endpoints $(k,v_\alpha(\rho_k))$ and $(k+l,v_\alpha(\rho_{k+l}))$ then
$
\RA_i(x)=\sum_{j=0}^{l}\underline{\rho_{j+k}\alpha^{j\slopenum-v_\alpha(\rho_k)}}\,x^{j}
=\underline{\rho(\alpha^\slopenum x)\alpha^{-k-v_\alpha(\rho_k)}\,x^{n-l}}
\in\RK[x].
$
\item \label{lem res pol end}
If for $1\le i\le\ell-1$ the leading coefficient of $\RA_i$ is denoted by $\RA_{i,\deg\RA_i}$ and
$\RA_{i+1,0}$ is the constant coefficient of $\RA_{i+1}$ then $\RA_{i,\deg\RA_i}=\RA_{i+1,0}$. 
\item If $\rho$ is monic then $\RA_\ell$ is monic.
\end{enumerate}
\end{lemma}

From now on we consider the residual polynomials of the segments of a ramification polygon.
We recall some of the results of \cite[Section 4]{pauli-sinclair}.

\begin{proposition}\label{prop res pol seg}
Let $f\in\OK[x]$ be Eisenstein of degree $n=p^r e_0$ with $\gcd(p,e_0)=1$, let $\alpha$ be a root of $f$, 
$\rho$ the ramification polynomial, and $\rpol_f$ the ramification polygon of $f$.
\begin{enumerate}
\item
If $e_0\ne 1$ then $\rpol_f$ has a horizontal
segment of length $p^r(e_0-1)$ with residual polynomial 
$\RA=\sum_{i=0}^{n-p^r}\RA_{i}x^i$ where 
$\RA_{i}=\underline{\binom{n}{i}}\ne\underline0$ if and only if  $v\binom{n}{i}=0$.
\item If $(p^{s_k},J_k),\dots,(p^{s_l},J_l)$ are the points on a segment $\spol$ of $\rpol_f$
of slope $-\frac{\slopenum}{\slopeden}$ of $\rpol_f$ then the residual polynomial of $\spol$ is
\[
\RA(x)
=
\sum_{i=k}^l\underline{\rho_{p^{s_i}}\alpha^{-J_i}}\,x^{(p^{s_i}-p^{s_k})/\slopeden}
=
\sum_{i=k}^l\underline{
f_{b_i} {\tbinom{b_i}{p^{s_i}}}\alpha^{-a_i n-n}
}\,x^{(p^{s_i}-p^{s_k})/\slopeden}.
\]
\end{enumerate}
\end{proposition}

This immediately gives:

\begin{corollary}\label{cor res pol seg}
Let $f\in\OK[x]$ be Eisenstein and 
$\rpol_f$ its ramification polygon of $f$.  
\begin{enumerate}
\item The residual polynomial of the rightmost segment of $\rpol$ is monic.
\item
Let $(p^{s_l},J_l)$ 
be the right end point of the $i$-th segment of $\rpol$ 
and $\RA_i=\sum_{j=0}^{m_i} \RA_{i,j}$ its residual polynomial
and let $(p^{s_k},J_k)$ 
be the left end point of the $(i+1)$-st segment of $\rpol$
and $\RA_{i+1}=\sum_{j=0}^{m_{i+1}} \RA_{i+1,j}$ its residual polynomial.
Then $\RA_{i,m_i}=\RA_{i+1,0}$.
\end{enumerate}
\end{corollary}

We give criteria for the existence of polynomials with given ramification polygon $\rpol$ with given residual polynomials of its segments.

\begin{proposition}\label{prop A}
Let $n=p^r e_0$ with $\gcd(p,e_0)=1$ and let
$\rpol$ be a ramification polygon with points
\[
\rpol = \{ (1,J_0),(p^{s_1},J_1),\dots,(p^{s_k},J_k),\dots,(p^{r},0),\dots,(p^re_0,0) \}.
\]
Write $J_k=a_kn+b_k$ with $0\le b_k\le n$.
Let $\spol_1,\dots,\spol_\ell$ be the segments of $\rpol$ with endpoints
$(p^{k_i},J_{k_i})$ and $(p^{l_i},J_{l_i})$
and slopes $-\slopenum_i/\slopeden_i$ ($1\le i <\ell$).
For $1\le i <\ell$ let
$\RA_i(x)=\sum_{j=0}^{(p^{l_i}-p^{k_i})/{\slopeden_i}}\RA_{i,j}x^j\in\RK$.

There is an Eisenstein polynomial $f$ of degree $p^r e_0$ with ramification polygon $\rpol$
and segments $\spol_1,\dots,\spol_\ell$ with residual polynomials
$\RA_1,\dots,\RA_\ell\in\RK[x]$ if and only if
\begin{enumerate}
\item $\RA_{i,\deg\RA_i}=\RA_{i+1,0}$ for $1\le i< \ell$,
\item $\RA_{i,j}\ne0$ if and only if $j=(q-p^{s_{k_i}})/\slopeden_i$ for some 
$q\in\{p^{s_1},\dots,p^r\}$ with  $p^{k_i}\le q\le p^{l_i}$,
\item if for some $1\le t,q\le u$ we have $b_t=b_q$ and 
$s_{k_i}\le s_t\le s_{l_i}$
and
$s_{k_j}\le s_q\le s_{l_j}$
then
\[
\RA_{i,(p^{s_t}-p^{s_{k_i}})/\slopeden_i}
=
\underline{
{\tbinom{b_t}{p^{s_t}}} {\tbinom{b_t}{p^{s_q}}}^{-1}
(-f_0)^{a_q-a_t}
}
\RA_{j,(p^{s_q}-p^{s_{k_j}})/\slopeden_j}.
\]
\end{enumerate}
\end{proposition}

We define the invariant $\invA$ of $\KL/\KK$
that is obtained from the residual polynomials of the 
segments of the ramification polygon of $f$.   
From the proof of \cite[Proposition 4.4]{greve-pauli} we obtain:

\begin{lemma}\label{lem res pol zeros}
Let $f\in\OK[x]$ be Eisenstein and $\alpha$ a root of $f$ and $\KL=\KK(\alpha)$.
Let $\spol$ be a segment of the ramification polygon of $f$ of slope $-\slopenum/\slopeden$
and let $\RA$ be its residual polynomial.
Let $\beta=\delta\alpha$ with $v_\alpha(\delta)=0$ be another uniformizer of $\KL$ 
and $g$ its minimal polynomial.
If $\underline\gamma_1,\dots,\underline\gamma_m$ are the (not necessarily distinct) zeros of $\RA$, then 
$\underline\gamma_1/\underline\delta^h,\dots,\underline\gamma_m/\underline\delta^h$ are the
zeros of the residual polynomial of the segment of slope $-\slopenum/\slopeden$
of the ramification polygon of $g$.
\end{lemma}

Thus the zeros of the residual polynomials of all  segments of the ramification polygon change by powers of the same 
element $\underline\delta$ when transitioning from a uniformizer $\alpha$ to a uniformizer $\delta\alpha$.
With Proposition \ref{prop A} we obtain: 
{
\begin{theorem}[{\cite[Theorem 4.8]{pauli-sinclair}}]\label{theo A}
Let $\spol_1,\dots,\spol_\ell$ be the segments of the ramification polygon 
$\rpol$ of an Eisenstein polynomial  $f\in\OK[x]$.
For $1\le i\le \ell$ let $-h_i/e_i$ be the slope of $\spol_i$
and $\RA_i(x)=\sum_{j=0}^{m_i}$ its residual polynomial.
Then 
\begin{equation}\label{eq A}
\invA=\left\{ 
\left(\gamma_{\delta,1}{\RA_1}(\underline\delta^{h_1} x),\dots,
\gamma_{\delta,\ell}{\RA_\ell}(\underline\delta^{h_\ell} x)\right) 
: 
\underline\delta\in\RK^\times
\right\}
\end{equation}
where 
$
\gamma_{\delta,\ell}=\delta^{-h_\ell\deg\RA_\ell},
$
and
$
\gamma_{\delta,i}=\gamma_{\delta,i+1}\delta^{-h_i\deg\RA_i}
$
for $1\le i\le \ell-1$
is an invariant of the extension $\KK[x]/(f)$. 
\end{theorem}
}

\begin{example}\label{ex theo A}
Let $f(x)=x^9+6x^3+9x+3$.  
The ramification polygon of $f$ consists of the two segments 
with end points $(1,10),(3,3)$ and $(3,3),(9,0)$
and residual polynomials $1+2x$ and $2+x^3$.
We get
\[
\invA=\{(1+2x,2+x^3),(1+x,1+x^3)\}.
\]
\end{example}

The choice of a representative of 
the invariant $\invA$ determines some of the coefficients of the generating polynomials with this invariant.

\begin{lemma}\label{lem A fix}
Let $f\in\OK[x]$ be Eisenstein of degree $n$. 
Let $\spol$ be a segment of 
ramification polygon
of $f$ with endpoints
$(p^{s_k},a_k n+b_k)$ and
$(p^{s_l},a_l n+b_l)$ and residual polynomial $\RA(x)=\sum_{j=1}^{p^{s_l}-p^{s_k}}\RA_j x^j\in\RK[x]$.
If $(p^{s_i},a_i n+b_i)$
is a point on $\spol$ with $b_i \neq 0$ then
\[
\underline{f}_{b_i,j}=
\RA_{(p^{s_i}-p^{s_k})/\slopeden}
\underline{
{ \tbinom{b_i}{p^{s_i}}}^{-1}(-f_{0,1})^{a_i+1}\pi^{v_\PP\binom{b_i}{p^{s_i}}} }
\]
where $j=a_i+1-v_\PP\binom{b_i}{p^{s_i}}$.
\end{lemma}

Finally, we construct a set of Eisenstein generating polynomials for extensions with given
degree, discriminant, ramification polygon, and invariant $\invA$.
This set $\Psi_{n,J_0,\rpol,\invA}(c)$ is a subset of $\Psi_{n,J_0,\rpol}(c)$, so its members
have the desired discriminant and ramification polygon, and by setting certain coefficients
of the $\PP$-adic expansion of polynomial coefficients, these polynomials will have
residual polynomials $(\RA_1,\ldots,\RA_\ell)\in\invA$ by construction.

\begin{proposition}\label{prop psi-resseg}
For a ramification polygon $\rpol$ with points $(p^0,J_0),(p^{s_1},J_1),\dots,\,(p^{s_\ell},J_\ell),$
where $J_i = a_in+b_i$, let $B_\rpol$ be the set of non-zero $b_i$.
Let $c > 1 + 2a_0 + \frac{2b_0}{n}$,
and let $\Psi_{n,J_0,\rpol,\invA}(c)$ be the set of all polynomials $\psi(x)=x^n + \sum \psi_i x^i \in \OK[x]$ with
\[
\psi_i \in \left\lbrace \begin{array}{ll}
\repss{1,c}          & \mbox{ if } i = 0\\
\repss{L_\rpol(i),c} & \mbox{ if } i \in B_\rpol \\
\reps{L_\rpol(i),c}  & \mbox{ if } 1\leq i\leq n-1 \mbox{ and } i \notin B_\rpol \\
\end{array}\right.
\]
and where all $\underline{\psi}_{i,L_\rpol(i)}$ for $i\in B_\rpol$ are set by the same choice of $(\RA_1,\ldots,\RA_\ell)\in\invA$
according to Lemma \ref{lem A fix}.
The polynomials in $\Psi_{n,J_0,\rpol,\invA}(c)$ generate totally ramified extensions of $\KK$ of degree $n$, discriminant $\PP^{n+J_0-1}$, ramification polygon $\rpol$, and invariant $\invA$.
\end{proposition}

%%%%%%%%%%%%%%%%%%%%%%%%%%%%%%%%%%%%%%%%%%%%%%%%%%%%%%%%%%%%%%%%%%%%%%%%%%%%%%

\section{An Ultrametric Distance of Polynomials}

For two irreducible polynomials $f,g\in\KK[x]$ of degree $n$, we use an ultrametric
distance defined by Krasner that we will later relate to the distance of the roots
of these two polynomials.

\begin{proposition}\label{prop dfg}
Let $f,g\in\KK[x]$ be two irreducible polynomials of degree $n$.
If $\alpha$ is a root of $f$ and $\beta$ is a root of $g$,
then $d(f,g) = |f(\beta)| = |g(\alpha)|$ defines an ultrametric distance
over the set of irreducible polynomials of degree $n$ in $\KK[x]$.
Additionally, if $\alpha=\alpha_1,\ldots,\alpha_n$ are the roots of $f$, and $\beta$ is one of the roots of $g$ which
is closest to $\alpha$, then
\[
d(f,g) = \prod_{i=1}^n \{ |\beta-\alpha|,|\alpha-\alpha_i| \}.
\]
\end{proposition}

We can calculate the distance $d(f,g)$ easily using the following lemma.

\begin{lemma}[\cite{pauli-roblot},Lemma 4.2]\label{lem dfg}
Using the same notation as Proposition \ref{prop dfg},
write $f(x) = x^n + f_{n-1}x^{n-1}+\cdots+f_0$ and $g(x) = x^n + g_{n-1}x^{n-1}+\cdots+g_0$, and set
\[
w = \min_{0\leq i\leq n-1} \left\{ v_\PP(g_i-f_i)+\frac{i}{n}\right\}.
\]
Then $d(f,g)=|\pi|^w$.
\end{lemma}

%%%%%%%%%%%%%%%%%%%%%%%%%%%%%%%%%%%%%%%%%%%%%%%%%%%%%%%%%%%%%%%%%%%%%%%%%%%%%%

\section{A Generalization of Krasner's Mass Formula}

We generalize the proof of Krasner's mass formula to include cases where additional invariants beyond the discriminant may be fixed.
Let $\invset$ be a set of invariants of a totally ramified extension over $\KK$ minimally containing a degree $n$ and discriminant $\PP^{n+J_0-1}$.

Let $\mathbf{K}_\invset$ denote the set of totally ramified extensions over $\KK$ with invariants $\invset$
and $\mathbf{E}_\invset$ denote the set of Eisenstein polynomials in $\KK[x]$ generating extensions with invariants $\invset$.
The roots of the polynomials in $\mathbf{E}_\invset$ generate all extensions in $\mathbf{K}_\invset$.
Let $c > 1+(2J_0)/n$ and $\Psi_\invset(c)$ be the set of all Eisenstein polynomials with coefficients in $R_{1,c}$ whose roots generate totally ramified extensions with invariants $\invset$.

\begin{theorem}[Krasner]\label{thm discs-D}
The set $\mathbf{E}_{n,J_0}$ of Eisenstein polynomials of degree $n$ and discriminant $\PP^{n+J_0-1}$ over $\KK$ is the disjoint union of the closed discs $D_{\mathbf{E}_{n,J_0}}(\psi,r)$ with centers $\psi\in\Psi_{n,J_0}(c)$ and radius $r=|\mathfrak{p}^c|$.
\end{theorem}

\begin{corollary}\label{cor discs-D}
Let $f$ be an Eisenstein polynomial of degree $n$ and discriminant $\mathfrak{p}^{n+J_0-1}$ over $\KK$
and write $f(x) = x^n + f_{n-1}x^{n-1}+\cdots+f_0$.
Let $g(x) = x^n + g_{n-1}x^{n-1}+\cdots+g_0$ be a polynomial such that $g_i \equiv f_i\mod\mathfrak{p}^c$.
Let $\alpha$ be a root of $f$ and $\beta$ a root of $g$ such that $|\beta-\alpha|$ is minimal.
Then $\alpha\in\KK(\beta)$.
\end{corollary}

Krasner's bound (Lemma \ref{lem krasner bound}) is a direct consequence of the previous corollary.

\begin{corollary} The set $\mathbf{E}_\invset$ is the disjoint union of the closed discs $D_{\mathbf{E}_\invset}(\psi,r)$ with centers $\psi\in\Psi_\invset(c)$ and radius $r=|\mathfrak{p}^c|$.
\end{corollary}

\begin{proof}
If $\invset$ only contains a degree and discriminant, then this is Theorem \ref{thm discs-D}.

By the definition of $\Psi_\invset$, for any $\psi\in\Psi_\invset(c)$, we have that $\psi\in \mathbf{E}_\invset$.
As $\invset$ contains a degree and discriminant, $\Psi_\invset(c) \subseteq \Psi_{n,J_0}(c)$, 
so $D_{\mathbf{E}_\invset}(r) \subseteq D_{\mathbf{E}_{n,J_0}}(r)$,
and so by Theorem \ref{thm discs-D}, 
we know that $\Psi_\invset(c)$ is a set of disjoint closed discs.

Let $f(x) = \sum f_i x^i \in \mathbf{E}_\invset$ and $g(x) = \sum \psi_i x^i$ such that
$g_i\in R_{1,c}$ and $g_i \equiv f_i \mod \mathfrak{p}^c$.
By Corollary \ref{cor discs-D}, $g\in\mathbf{E}_\invset$, and with $g_i\in R_{1,c}$, we have $g\in\Psi_\invset$.
By our choices of $g_i$, we have that $v_\PP(f_i-g_i) \geq c$ for $i=0,\ldots,n-1$.
Therefore, for all $i$,
\[ v_\PP(f_i-g_i) + \frac{i}{n} \geq c \]
which, by Lemma \ref{lem dfg}, shows that $f\in D_{\mathbf{E}_\invset}(g,r)$ with $r=|\mathfrak{p}^c|$.
\end{proof}

\begin{lemma} Let $\invset$ be a set of invariants of a totally ramified extension over $\KK$ containing degree $n$ and discriminant $\PP^{n+J_0-1}$.
Let $c > 1 + (2J_0)/n$ and
let $\#D_{\mathbf{E}_\invset}(r)$ denote the number of disjoint closed discs of radius $r=|\pi^c|$ in $\mathbf{E}_\invset$.
Then the number of elements in $\mathbf{K}_\invset$ is
\[
\#\mathbf{K}_\invset = \#D_{\mathbf{E}_\invset}(r)\frac{n}{(q-1)q^{nc-(n+J_0-1)-2}}
\]
\end{lemma}
\begin{proof}
Let $\Pi_\invset = \cup_{L\in\mathbf{K}_\invset}\mathfrak{p}_\KL\setminus\mathfrak{p}_\KL^2$, where $\mathfrak{p}_\KL$ is the prime ideal of $L$.
Let $\chi$ be the map that sends $\alpha\in\Pi_\invset$ to its minimal polynomial $\chi(\alpha) \in \mathbf{E}_\invset$.

Let $t>J_0+1$ be an integer and let $s = |\pi^{(n+j_0-1+t)/n}|$.
Let $u = |\pi^t|^{1/n}$, and let $\alpha,\beta\in \Pi_\invset$ such that $|\alpha-\beta|\leq u$.
By Krasner's Lemma, $\alpha$ and $\beta$ generate the same field.
Let $\alpha=\alpha_1,\alpha_2,\ldots,\alpha_n$ denote the roots of $\chi(\alpha)$.
Then
\begin{align*}
d(\chi(\alpha),\chi(\beta)) = &\prod_{i=1}^n \max \{ |\beta-\alpha|, |\alpha-\alpha_i| \}\\
 &\leq u \prod_{i=2}^n |\alpha-\alpha_i|
 = u |(\chi(\alpha))'(\alpha)|
 = u |\pi^{(n+j-1)/n}| = s
\end{align*}
Let $D_\Pi(\alpha,u)$ denote the closed disc of center $\alpha$ and radius $u$ in $\Pi_\invset$.
As $d(\chi(\alpha),\chi(\beta)) \leq s$, we have $\chi(D_\Pi(\alpha,u)) \subset D_{\mathbf{E}_\invset}(\chi(\alpha),s)$.
Conversely, let $f,g\in\mathbf{E}_\invset$ such that $d(f,g)\leq s$.
Let $\alpha$ be a root of $f$ so $f=\chi(\alpha)$ and $\beta$ be the root of $g$ such that $|\beta-\alpha|$ is minimal.
We claim that $|\beta-\alpha| < u$, as otherwise
\begin{align*}
d(f,g) = &\prod_{i=1}^n \max \{|\beta-\alpha|,|\alpha-\alpha_i|\}
 \geq \prod_{i=1}^n \max \{u,|\alpha-\alpha_i|\}\\
 &\geq u \prod_{i=1}^n |\alpha-\alpha_i|
 = u |f'(\alpha)|
 = u |\pi^{(n+j-1)/n}| = s,
\end{align*}
which contradicts the assumption that $d(f,g)<s$.
As $|\beta-\alpha| < u$, we have $D_{\mathbf{E}_\invset}(\chi(\alpha),s) \subset \chi(D_\Pi(\alpha,u))$.
So, for all $\alpha\in\Pi_\invset$,
\[
D_{\mathbf{E}_\invset}(\chi(\alpha),s)=\chi(D_\Pi(\alpha,u)).
\]
It is clear that the map $\chi$ is $n$-to-one and surjective.
Now, the inverse image of $\chi(\alpha)$ is the set of conjugates of $\alpha$ over $\KK$.
As $t>j+1$, the closed discs of radius $u$ centered at these conjugates are all disjoint.
Thus, the inverse image of any closed disc of radius $s$ in $\mathbf{E}_\invset$ is the
disjoint union of $n$ closed discs of radius $u$ in $\Pi_\invset$.
However, by the earlier remark, any such disc is contained in
$\mathfrak{p}_\KL\setminus\mathfrak{p}_\KL^2$ for some $\KL \in \mathbf{K}_\invset$.
Therefore, the number of disjoint closed discs of radius $u$ in $\Pi_\invset$ is equal to
$\#\mathbf{K}_\invset$ times the number of disjoint closed discs in $\mathfrak{p}_\KL\setminus\mathfrak{p}_\KL^2$,
which does not depend on $\KL$ and is $q^{t-1}-q^{t-2}$.
Thus,
\[
\#\mathbf{K}_\invset\; q^{t-2}(q-1) = n\; \#D_{\mathbf{E}_\invset}(s),
\]
and choosing $t = nc-(n+J_0-1)$ gives us our result.
\end{proof}

%%%%%%%%%%%%%%%%%%%%%%%%%%%%%%%%%%%%%%%%%%%%%%%%%%%%%%%%%%%%%%%%%%%%%%%%%%%%%%

\section{Mass Formula Given a Discriminant (Krasner)}

\begin{proposition}
Let $\Psi_{n,J_0}$ be the set of polynomials over $\KK$ with degree $n$ and discriminant $\PP^{n+J_0-1}$ whose coefficients are in $R_{1,c}$.
The number of polynomials in $\Psi_{n,J_0}$ is
\[
\#\mathbf{D}_{E_{n,J_0}}(c) = \left\{\begin{array}{ll}
 (q-1)   \, q^{ c-2+(n-1)c-\sum_{i=1}^{n-1} l(i)}  & \mbox{ for } b = 0\\
 (q-1)^2 \, q^{ c-2+(n-1)c-\sum_{i=1}^{n-1} l(i)-1} & \mbox{ for } b > 0
\end{array}\right.
\]
\end{proposition}

\begin{theorem}
The number of distinct totally ramified extensions of $\KK$ of degree $n$ and discriminant
$\PP^{n+J_0-1}$ is
\[
\#\mathbf{K}_{n,J_0} = \left\{\begin{array}{ll}
 n          \, q^{ n+J_0-1-\sum_{i=1}^{n-1} l(i)}  & \mbox{ for } b = 0\\
 n \, (q-1) \, q^{ n+J_0-1-\sum_{i=1}^{n-1} l(i)-1}  & \mbox{ for } b > 0
\end{array}\right.
\]
\end{theorem}

\begin{example}\label{ex cnt disc}
As an example, let us count all totally ramified extensions of $\Q_3$ with degree 9 and discriminant $(3)^{9+7-1}$.
From this discriminant, we have $J_0=7$.  We find minima for the $v_\PP(\varphi_i)$ if $\varphi$ is to be an
Eisenstein polynomial of this discriminant. By Lemma \ref{lem ell},
\[l(i) = v_\PP(\varphi_i) \geq \left\{\begin{array}{ll}
 2  & \mbox{ for } i \in \{1,2,4,5\} \\
 1  & \mbox{ for } i \in \{3,6,7,8\}
\end{array}\right.\]
So, $\sum l(i) = 12$, and from the formula, we find that there are
$9 \cdot 2 \cdot 3^{9+7-1-12-1} = 162$
degree 9 extensions of $\Q_3$ with discriminant $(3)^{9+7-1}$.
\end{example}

%%%%%%%%%%%%%%%%%%%%%%%%%%%%%%%%%%%%%%%%%%%%%%%%%%%%%%%%%%%%%%%%%%%%%%%%%%%%%%

\section{Mass Formula Given a Ramification Polygon}

\begin{proposition}
Let $\Psi_{n,J_0,\rpol}(c)$ be the set of Eisenstein polynomials with degree $n$, discriminant $\PP^{n+J_0-1}$, and ramification polygon $\rpol$ with coefficients whose coefficients above $c$ are zero (see Lemma \ref{lem krasner bound}). Then
%The number of $\psi\in\Psi_{n,J_0,\rpol}(c)$ is given by
\[
%\#\mathbf{D}_{\Psi_{n,J_0,\rpol}}(c) = (q-1)^{\#B_\rpol+1} \, q^{c-2+(n-1)c-\sum_{i=1}^{n-1} L(i)-\#B_\rpol}
\#\Psi_{n,J_0,\rpol}(c) = (q-1)^{\#B_\rpol+1} \, q^{c-2+(n-1)c-\sum_{i=1}^{n-1} L(i)-\#B_\rpol}
\]
\end{proposition}
\begin{proof}
The number of elements in $R^*_{1, c}$ is $(q-1) \, q^{c-2}$.  For each $i\notin B_\rpol$, the number of elements in $R_{L_\rpol(i),c}$ is $q^{c-L(i)}$, and for $i\in B_\rpol$ the number in $R_{L_\rpol(i),c}^*$ is $(q-1)q^{c-L(i)-1}$. The product of these is our result.
\end{proof}

\begin{theorem}
The number of distinct totally ramified extensions of $\KK$ of degree $n$, discriminant
$\PP^{n+J_0-1}$, and ramification polygon $\rpol$ is
\[
n(q-1)^{ \#B_\rpol} \, q^{ n+J_0-1 - \sum_{i=1}^{n-1} L(i)-\#B_\rpol}
\]
\end{theorem}
\begin{proof}
\[
\frac{n \; \#\mathbf{D}_{E_{n,J_0,\rpol}}(c)}{(q-1)q^{nc-(n+J_0-1)-2}} = n(q-1)^{\#B_\rpol} \, q^{n+J_0-1 - \sum_{i=0}^{n-1} L(i)-\#B_\rpol}
\]
\end{proof}

\begin{example}[Example \ref{ex cnt disc} continued]\label{ex cnt rp}
We count the totally ramified extensions of $\Q_3$ with degree 9 and discriminant $(3)^{9+7-1}$
for the possible choices of ramification polygons.  Again, we have $J_0=7$ and
\[l_\rpol(i,0) = l(i) = \left\{\begin{array}{ll}
 2  & \mbox{ for } i \in \{1,2,4,5\} \\
 1  & \mbox{ for } i \in \{3,6,7,8\}
\end{array}\right.\]

There are two possible ramification polygons for this degree and discriminant: $\rpol_1$ with vertices $\{(1,7),(9,0)\}$ and $\rpol_2$ with vertices $\{(1,7),(3,3),(9,0)\}$.  We have already considered the conditions on the polynomial dictated by the vertex $(1,7)$, so it only remains to consider the effect of a vertex (or lack thereof) above 3.

For $\rpol_1$, no vertex above 3 means $l_{\rpol_1}(3,1) = 2$ and $l_{\rpol_1}(6,1) = 1$.
For an Eisenstein polynomial to have ramification polynomial $\rpol_1$, the minimum valuations of the coefficients are
\[L_{\rpol_1}(i) = \max_s \{l_{\rpol_1}(i,s)\} = \left\{\begin{array}{ll}
 2  & \mbox{ for } i \in \{1,2,3,4,5\} \\
 1  & \mbox{ for } i \in \{6,7,8\}
\end{array}\right..\]
So, $\sum L_{\rpol_1}(i) = 13$.
Next we find the set of indexes ($1\leq i\leq n-1$) of coefficients that must have a fixed valuation in order for an Eisenstein polynomial to generating such an extension according to Proposition \ref{prop R iff}.
\[ B_{\rpol_1} = \{J_i\text{ mod } n : 0\leq i \leq s_\ell \text{ and } J_i\text{ mod } n\neq 0 \} = \{7\} \]
As $\#B_{\rpol_1} = 1$, only one coefficient must have its minimum valuation.
Thus, by applying the formula, we find that there are $9 \cdot 2^1 \cdot 3^{9+7-1-13-1} = 54$ degree 9 extensions of $\Q_3$ with ramification polygon $\rpol_1$.

For $\rpol_2$, the vertex $(3,3)$ gives us that $l_{\rpol_2}(3,1) = 1$ and $l_{\rpol_2}(6,1) = 1$.
For an Eisenstein polynomial to have ramification polynomial $\rpol_2$, the minimum valuations of the coefficients would have to be
\[L_{\rpol_2}(i) = \max_s \{l_{\rpol_2}(i,s)\} = \left\{\begin{array}{ll}
 2  & \mbox{ for } i \in \{1,2,4,5\} \\
 1  & \mbox{ for } i \in \{3,6,7,8\}
\end{array}\right. .\]
So, $\sum L_{\rpol_2}(i) = 12$.
The set of indexes ($1\leq i\leq n-1$) of coefficients that must have a specific valuation in order for an Eisenstein polynomial to generating such an extension is
\[ B_{\rpol_2} = \{J_i\text{ mod } n : 0\leq i \leq s_\ell \text{ and } J_i\text{ mod } n\neq 0\} = \{3,7\}. \]
So there are $\#B_{\rpol_2} = 2$ coefficients of fixed valuation.
Thus, by applying the formula, we find that there are $9 \cdot 2^2 \cdot 3^{9+7-1-12-2} = 108$ degree 9 extensions of $\Q_3$ with ramification polygon $\rpol_2$.

Krasner's mass formula states that there are 162 totally ramified extensions of $\Q_3$ with degree 9, which we have partitioned by the two possible ramification polygons.
\end{example}

%%%%%%%%%%%%%%%%%%%%%%%%%%%%%%%%%%%%%%%%%%%%%%%%%%%%%%%%%%%%%%%%%%%%%%%%%%%%%%

\section{Mass Formula Given Residual Polynomials}

\begin{proposition}
The number of Eisenstein polynomials of degree $n$, with given discriminant $\PP^{n+J_0-1}$, ramification polygon $\rpol$, and invariant $\invA$ with coefficients in $\reps{1,c}$ is
\[
(\#\invA)\, (q-1) \, q^{c-2+(n-1)c-\sum_{i=1}^{n-1} L_\rpol(i)-\#B_\rpol}
%(\#\invA) \, q^{c-2+(n-1)c-\sum_{i=1}^{n-1} L_\rpol(i)-\#B_\rpol}
\]
\end{proposition}
\begin{proof}
The first non-zero coefficient of the constant term of an Eisenstein polynomial is not determined by the choice of $\invA$, so that coefficient may be any element of $R^*_{1, c}$, of which there are $(q-1) \, q^{c-2}$ elements.
For each $i\notin B_\rpol$, we have the number of elements in $R_{L_\rpol(i),c}$, which is $q^{c-L(i)}$.
For $i\in B_\rpol$, the choice of $(\RA_1,\ldots,\RA_\ell) \in \invA$, fixes the first non-zero coefficient of our coefficients.
The number of elements in $R_{L_\rpol(i),c}^*$ with a fixed first non-zero coefficient is $q^{c-L(i)-1}$.
We have $\#\invA$ ways to fix those coefficients, and the product of these is our result.
\end{proof}

\begin{theorem}
The number of distinct totally ramified extensions of $\KK$ of degree $n$, discriminant
$\PP^{n+J_0-1}$, ramification polygon $\rpol$, and invariant $\invA$ is
\[
n \, (\#\invA)\, q^{n+J_0-1 - \sum_{i=1}^{n-1} L_\rpol(i)-\#B_\rpol}
%n \, \frac{(\#\invA)}{(q-1)}\, q^{n+J_0-1 - \sum_{i=1}^{n-1} L_\rpol(i)-\#B_\rpol}
\]
\end{theorem}

\begin{example}[Example \ref{ex cnt rp} continued]
We count all totally ramified extensions of $\Q_3$ with degree 9, discriminant $(3)^{9+7-1}$,
and ramification polygon $\rpol_2 = \{(1,7),(3,3),(9,0)\}$

As before, for an Eisenstein polynomial to have ramification polynomial $\rpol_2$, the minimum valuations of the coefficients would have to be
\[L_{\rpol_2}(i) = \max_s \{l_{\rpol_2}(i,s)\} = \left\{\begin{array}{ll}
 2  & \mbox{ for } i \in \{1,2,4,5\} \\
 1  & \mbox{ for } i \in \{3,6,7,8\}
\end{array}\right. .\]
So, $\sum_{i=1}^8 L_{\rpol_2}(i) = 12$ and the number of coefficients of fixed valuation is $\#B_{\rpol_2} = 2$.

There are four possible sets of residual polynomials of segments $(\RA_1,\RA_2)$ for extensions with ramification polygon $\rpol_2$, belonging to two invariants $\invA$:
\[
\invA_1 = \{ (x^2 + 1, x^3 + 1), (2x^2 + 2, x^3 + 2) \}
\]\[
\invA_2 = \{ (x^2 + 2, x^3 + 1), (2x^2 + 1, x^3 + 2) \}
\]
Each of these invariants contain two polynomials, so by applying the formula,
we find that there are $9 \cdot 2\cdot 3^{9+7-1-12-2} = 54$ degree 9 extensions of $\Q_3$
with ramification polygon $\rpol_2$ and a choice of $\invA$.

\end{example}

%%%%%%%%%%%%%%%%%%%%%%%%%%%%%%%%%%%%%%%%%%%%%%%%%%%%%%%%%%%%%%%%%%%%%%%%%%%%%%

\section{Examples}\label{sec counting ex}

In Table \ref{tab counting ex}, we show the number of extensions of degree 9 over $\Q_3$ with given invariants.
For discriminants $(3)^{9+J_0-1}$ with $J_0 \leq 12$, we list all possible ramification polygons and invariants $\invA$, and how many extensions exist with each set of invariants.

Additional examples for different base fields and degrees with all possible discriminants can be found at
\begin{center}
\url{http://www.uncg.edu/mat/numbertheory/tables/local/counting/}
\end{center}

\vfill

\renewcommand{\arraystretch}{1.1}
\begin{table}
\caption{Number of extensions of degree 9 for all possible ramification polygons and residual polynomials over $\Q_3$ with discriminant $(3)^{9+J_0-1}$ for $J_0 \leq 12$.}
\begin{center}

\begin{tabular}{|c|c|c|c|c|c|c|}\hline

        $J_0$ 
        &Ramification Polygon 
        &Representative of $\invA$%Residual Polynomials 
        &$\#\mathcal{A}$ 
        &\multicolumn{3}{|c|}{Extensions}\\
\hline
\multirow{1}{*}{1}&\multirow{1}{*}{$\{(1,1), (9,0)\}$}&$(z + 
1)$&2&18&\multirow{1}{*}{18}&\multirow{1}{*}{18}\\
\hline
\multirow{2}{*}{2}&\multirow{2}{*}{$\{(1,2), (9,0)\}$}&$(z^2 + 
1)$&1&9&\multirow{2}{*}{18}&\multirow{2}{*}{18}\\
\cline{3-5}
& &$(z^2 + 2)$&1&9& & \\
\hline
\multirow{4}{*}{4}&\multirow{2}{*}{$\{(1,4), (9,0)\}$}&$(z^4 + 
1)$&1&9&\multirow{2}{*}{18}&\multirow{4}{*}{54}\\
\cline{3-5}
& &$(z^4 + 2)$&1&9& & \\
\cline{2-6}
&\multirow{2}{*}{$\{(1,4), (3,3), (9,0)\}$}&$(z^4 + z + 
1)$&2&18&\multirow{2}{*}{36}& \\
\cline{3-5}
& &$(z^4 + z + 2)$&2&18& & \\
\hline
\multirow{3}{*}{5}&\multirow{1}{*}{$\{(1,5), (9,0)\}$}&$(z + 
1)$&2&18&\multirow{1}{*}{18}&\multirow{3}{*}{54}\\
\cline{2-6}
&\multirow{2}{*}{$\{(1,5), (3,3), (9,0)\}$}&$(z^2 + 1, z^3 + 
1)$&2&18&\multirow{2}{*}{36}& \\
\cline{3-5}
& &$(2z^2 + 1, z^3 + 2)$&2&18& & \\
\hline
\multirow{3}{*}{7}&\multirow{1}{*}{$\{(1,7), (9,0)\}$}&$(z + 
1)$&2&54&\multirow{1}{*}{54}&\multirow{3}{*}{162}\\
\cline{2-6}
&\multirow{2}{*}{$\{(1,7), (3,3), (9,0)\}$}&$(z^2 + 1, z^3 + 
1)$&2&54&\multirow{2}{*}{108}& \\
\cline{3-5}
& &$(2z^2 + 1, z^3 + 2)$&2&54& & \\
\hline
\multirow{8}{*}{8}&\multirow{2}{*}{$\{(1,8), (9,0)\}$}&$(z^8 + 
1)$&1&9&\multirow{2}{*}{18}&\multirow{8}{*}{162}\\
\cline{3-5}
& &$(z^8 + 2)$&1&9& & \\
\cline{2-6}
&\multirow{2}{*}{$\{(1,8), (3,3), (9,0)\}$}&$(z + 1, z^3 + 
1)$&2&54&\multirow{2}{*}{108}& \\
\cline{3-5}
& &$(z + 2, z^3 + 1)$&2&54& & \\
\cline{2-6}
&\multirow{4}{*}{$\{(1,8), (3,6), (9,0)\}$}&$(z^8 + z^2 + 
1)$&1&9&\multirow{4}{*}{36}& \\
\cline{3-5}
& &$(z^8 + 2z^2 + 1)$&1&9& & \\
\cline{3-5}
& &$(z^8 + z^2 + 2)$&1&9& & \\
\cline{3-5}
& &$(z^8 + 2z^2 + 2)$&1&9& & \\
\hline
\multirow{8}{*}{10}&\multirow{2}{*}{$\{(1,10), (9,0)\}$}&$(z^2 + 
1)$&1&27&\multirow{2}{*}{54}&\multirow{8}{*}{486}\\
\cline{3-5}
& &$(z^2 + 2)$&1&27& & \\
\cline{2-6}
&\multirow{2}{*}{$\{(1,10), (3,3), (9,0)\}$}&$(z + 1, z^3 + 
1)$&2&162&\multirow{2}{*}{324}& \\
\cline{3-5}
& &$(z + 2, z^3 + 1)$&2&162& & \\
\cline{2-6}
&\multirow{4}{*}{$\{(1,10), (3,6), (9,0)\}$}&$(z^2 + 1, z^6 + 
1)$&1&27&\multirow{4}{*}{108}& \\
\cline{3-5}
& &$(2z^2 + 1, z^6 + 2)$&1&27& & \\
\cline{3-5}
& &$(z^2 + 2, z^6 + 1)$&1&27& & \\
\cline{3-5}
& &$(2z^2 + 2, z^6 + 2)$&1&27& & \\
\hline
\multirow{5}{*}{11}&\multirow{1}{*}{$\{(1,11), (9,0)\}$}&$(z + 
1)$&2&54&\multirow{1}{*}{54}&\multirow{5}{*}{486}\\
\cline{2-6}
&\multirow{2}{*}{$\{(1,11), (3,3), (9,0)\}$}&$(z^2 + 1, z^3 + 
1)$&2&162&\multirow{2}{*}{324}& \\
\cline{3-5}
& &$(2z^2 + 1, z^3 + 2)$&2&162& & \\
\cline{2-6}
&\multirow{2}{*}{$\{(1,11), (3,6), (9,0)\}$}&$(z + 1, z^6 + 
1)$&2&54&\multirow{2}{*}{108}& \\
\cline{3-5}
& &$(2z + 1, z^6 + 2)$&2&54& & \\
\hline
\multirow{2}{*}{12}&\multirow{2}{*}{$\{(1,12), (3,3), (9,0)\}$}&$(2z + 1, z^3 +
2)$&1&243&\multirow{2}{*}{486}&\multirow{2}{*}{486}\\
\cline{3-5}
& &$(z + 2, z^3 + 1)$&1&243& & \\
\hline
\end{tabular}
\end{center}
\label{tab counting ex}
\end{table}

%%%%%%%%%%%%%%%%%%%%%%%%%%%%%%%%%%%%%%%%%%%%%%%%%%%%%%%%%%%%%%%%%%%%%%%%%%%%%%

\bibliography{local}
\bibliographystyle{amsalpha}
\end{document}